\documentclass[11pt]{article}

\usepackage{amsmath}
\usepackage{amsthm}
\usepackage{amsfonts}
\usepackage{setspace}
\usepackage{fullpage}
\usepackage{amssymb}
\usepackage{enumerate}

\usepackage{comment}
\newtheorem{lemma}{Lemma}
\newtheorem{theorem}{Theorem} 
 
\newtheorem{corollary}{Corollary}

\newtheorem{claim}{Claim}

\date{}

\newcounter{num}

\title{Sum-distinguishing number of sparse hypergraphs}
\author{Maria Axenovich\thanks{Department of Mathematics, Karlsruhe Institute of Technology, Germany},  ~Yair Caro\thanks{Department of Mathematics, University of Haifa-Oranim, Israel}, and Raphael Yuster \thanks{Department of Mathematics, University of Haifa, Israel}}

\begin{document}
\maketitle

\begin{abstract}

A vertex labeling of a hypergraph is {\em sum distinguishing} if it uses positive integers and the sums of labels taken over the distinct hyperedges are distinct.
Let $s(H)$ be the smallest integer $N$ such that there is a sum-distinguishing labeling of $H$ with each label at most $N$. The largest value of $s(H)$ over all hypergraphs on $n$ vertices and $m$ hyperedges is denoted $s(n,m)$.
We prove that $s(n,m)$ is almost-quadratic in $m$ as long as $m$ is not too large.
More precisely, the following holds: If $n \le m \le n^{O(1)}$ then
$$
s(n,m)=\frac{m^2}{w(m)},
$$
where $w(m)$ is a function that goes to infinity and is smaller than any polynomial in $m$.

The parameter $s(n,m)$ has close connections to several other graph and hypergraph functions, such as the irregularity strength of hypergraphs. Our result has several applications, notably:
\begin{itemize}
	\item
	We answer a question of Gy\'arf\'as et al. whether there are $n$-vertex hypergraphs with irregularity strength greater than $2n$. In fact we show that there are $n$-vertex hypergraphs with irregularity strength at least $n^{2-o(1)}$.
	\item
	In addition, our results imply that
	$s^*(n)=n^2/w(n)$ where $s^*(n)$ is the distinguishing closed-neighborhood number, i.e., the smallest integer $N$ such that any $n$-vertex graph allows for a vertex labeling with positive integers at most $N$ so that the sums of labels on distinct closed neighborhoods of vertices are distinct.
\end{itemize}

\end{abstract}

\section{Introduction}


For a hypergraph $H=(V,E)$, we say that a labeling $f: V\rightarrow \mathbb{N}$ is {\it   sum-distinguishing}  or simply {\it distinguishing}  if $s(e)\neq s(e')$ for any two distinct hyperedges $e, e'\in E$, where $s(e) = \sum_{v\in e} f(v)$. Let $s(H)$ be the smallest integer $N$ such that there is a distinguishing labeling of $H$ with each label at most $N$.
Note that $s(H)$ is well-defined by assigning vertex labels equal to distinct powers of $2$.
Distinguishing labelings can be viewed as number-theoretic constructions extending Sidon sets to non-complete, non-uniform hypergraphs.
Using common notation, a $B_h[1]$-Sidon set is a set $X$ of integers such that for any integer $q$, there is at most one subset $X'$ of $X$, $|X'|=h$,  so that the sum of elements from $X'$ is $q$. So, a $B_h[1]$-Sidon set corresponds to a distinguishing labeling of a complete $h$-uniform hypergraph.  On the other hand, distinguishing labelings of hypergraphs are closely connected to several ``distinguishing'' type parameters of graphs and hypergraphs that we discuss in more detail later. Let 
$$
s(n,m)=\max\{ s(H)\,:\,   |V(H)|=n, |E(H)|=m\}\;.
$$
Namely, $s(n,m)$ is the largest value of $s(H)$ over all hypergraphs on $n$ vertices and $m$ hyperedges.
Observe first that for the largest possible value of $m$, namely $m=2^n-1$ (corresponding to the full hypergraph consisting of all possible hyperedges), it trivially holds that $s(n,2^n-1)=2^{n-1}$.
So, in particular, we have that $s(n,m)$ is linear in the number of edges whenever $m=\Theta(2^n)$.
On the other hand, for general $m$, a standard probabilistic argument shows that $s(n,m)=O(m^2)$. So, it seems of interest to study the dependence of $s(n,m)$ on $m$ whenever the hypergraph is relatively sparse.
Our main result does just that. We prove, perhaps surprisingly, that for hypergraphs with polynomially many edges, $s(n,m)$ is neither linear nor quadratic. In fact, we prove that in this regime, $s(n,m)$ is almost-quadratic in $m$.
\begin{theorem}\label{t:snm}
If $n \le m \le n^{O(1)}$ then
$$
s(n,m)=\frac{m^2}{w(m)}\;,
$$
where $w(m)$ is a function that goes to infinity and is smaller than any polynomial in $m$.
More formally, for any $C>0$, $\epsilon > 0$, there is $n_0$ such that for any $n>n_0$, and any $m$ satisfying $n\leq m \leq n^{1/\epsilon}$, we have that $m^{2-\epsilon} \leq s(n,m) \leq m^2/C$.
\end{theorem}
The upper bound in the proof of Theorem \ref{t:snm} relies on several probabilistic arguments, some of which are rather delicate. For the lower bound, we extend an approach of Bollob\'as and Pikhurko \cite{BP} used for $2$-uniform hypergraphs (i.e. graphs) and their distinguishing labelings.

Our main result has several applications that we next describe.
Our first application is to the problem of distinguishing the vertices of a graph by sums of labels on closed neighborhoods. For a graph $G=(V,E)$, and a vertex $v\in V(G)$, the open neighborhood of $v $ is
$N(v)=\{u\in V(G)\,:\,  uv\in E(G)\}$; the  closed neighborhood of $v$ is  $N[v]=\{v\} \cup N(v)$.
For a vertex labeling $f$ of $G$ and $v\in V(G)$, let $s^*(v) =s^*_f(v)= \sum_{x\in N[v]}  f(x)$.
The labeling $f$ is  called {\it vertex sum-distinguishing} if it uses positive integers and 
$s^*(v)\neq s^*(u)$ for any $u, v\in V(G)$ such that $N[u] \neq N[v]$.
Let $s^*(G)$ be the smallest integer $k$ such that there is a vertex sum-distinguishing labeling of $G$ with a largest label $k$ and let $s^*(n)$ be the maximum of $s^*(G)$ taken over all graphs with $n$ vertices.

Let $s(n)=s(n,n)$. We observe that the parameters $s(n)$ and $s^*(n)$ are closely connected.  
Indeed, for a graph $G=(V, E)$ consider a hypergraph $H=H_G$ on a  vertex set $V$ with hyperedges corresponding to the closed neighbourhoods of vertices of $G$. We see that $s^*(G)  = s(H)$. Note that the number of hyperedges in $H_G$ is at most $n$. The following result is an immediate consequence of Theorem \ref{t:snm}
and Lemma \ref{l:equivalence}, in which we prove that $s(n/2) \leq s^*(n) \leq s(n)$.
Thus, we obtain:
\begin{corollary}\label{c:max-graphs} We have that 
	$$
	s^*(n)=\frac{n^2}{w(n)}\;,
	$$
	where $w(n)$ is a function that goes to infinity and is smaller than any polynomial in $n$. More formally, for any $C>0$, $\epsilon > 0$, there is $n_0$ such that for any $n>n_0$,
	$n^{2-\epsilon} \leq s^*(n) \leq n^2/C$. 
\end{corollary}
The proof of Theorem \ref{t:snm} (and hence Corollary \ref{c:max-graphs})  yields an efficient randomized algorithm for finding a corresponding labeling. For graphs with given maximum and minimum degrees, we provide a more specific result which also yields an efficient deterministic algorithm.
\begin{theorem}\label{t:Delta}
Let $G$ be a nonempty $n$-vertex graph with maximum degree $\Delta$, minimum degree $\delta$, and the largest number of vertices with pairwise distinct closed neighborhoods equal to  $n'$. Let $d(v)$ denote the degree of $v$. 
Then
$$
\frac{n'+ \delta}{\Delta +1} \leq s^*(G) \leq  \max \{(n-d(v)-1)(d(v)+1)+2\,:\,v\in V(G)\}\leq (\Delta+1) n\;.
$$
\end{theorem}
Observe that the upper bound of Theorem \ref{t:Delta} is weaker than the upper bound in Corollary \ref{c:max-graphs} whenever $\Delta=\Theta(n)$. In fact, it only gives $s^*(n)\le n^2/4+2$.
As a final result concerning $s^*$, we consider the case where the graph is a tree.
For a tree $T$ and its vertex $u$, let $L(u)$ be the set of leaves adjacent to $u$. Let
$L(T) = \max \{|L(u)|\,:\, u \in V(T)\}$.
\begin{theorem}\label{t:trees}
Let $T$ be a tree with $n\geq 3$ vertices. Then $s^*(T) \leq  2n-2-L(T)$, moreover this bound is tight for stars.  
\end{theorem}

The parameter $s(n)$ is closely related to the notion of {\it irregularity strength} of hypergraphs. There is extensive literature on irregularity strengths of graphs, a notion first introduced (for graphs) by Chartrand et al. \cite{C}, see also for example Nierhoff \cite{N} and  Blokhuis and Sz\H{o}nyi \cite{BSz},  Balister et al. \cite{BBLM}, as well as the survey by Gallian \cite{Ga}.  To define irregularity strength, consider an edge-labeling $f$ of a hypergraph $H$ with positive integers and for each vertex $x$ compute $s(x)$, the sum of labels over all hyperedges containing $x$. The labeling is {\it irregular}, if the sums $s$ are distinct for all vertices. The smallest value  of a largest label used in an irregular labeling of $H$ is denoted $irr(H)$ and $irr(n)$ is the largest value of $irr(H)$ over all $n$-vertex hypergraphs. Note that $irr(H)$ corresponds to $s(H^*)$, where $H^*$ is the dual hypergraph of $H$. 
Recall that for a hypergraph $H=(V, E)$, the dual hypergraph $H^*$ has vertex set $E$ and edge set $\{ \{  e\in E:   e \ni x\}:  x\in V\}$.
Gy\'arf\'as et al. \cite{G} provided upper bounds on $irr(H)$ and stated  that ``it is not known whether $irr(n)\geq 2n$.''
A consequence of our result gives a better lower bound $irr(n) \geq  n^{2- \epsilon}$ for any positive $\epsilon$ and sufficiently large $n$, and in particular, answers their question.
\begin{theorem}\label{t:irregular}
For any $\epsilon>0$, there is $n_0$ such that for any  $n>n_0$,  $irr(n)\geq n^{2 - \epsilon}$. 
\end{theorem} 

We mention a few other closely related problems that have been studied.
There is yet another parameter, similar to $s(H)$,  introduced by Bhattacharya et al. \cite{BDG} and  called a {\it discriminator} where the goal is to assign non-negative integer labels to vertices of a hypergraph such that the sums on the hyperedges are distinct, and positive. 
While our original motivation was  to  distinguish the vertices of a graph via sums on closed neighborhoods, there is a similar problem restricted to pairs of vertices that are adjacent, i.e., so-called adjacent vertex sum-distinguishing number, that was studied for closed neighborhoods by Axenovich et al. \cite{AHPSVW} and for open neighborhoods by Bartnicki et al. \cite {BB}, who use also an unpublished observation by  Norin. These above-mentioned adjacency-dependent parameters can however be upper-bounded by a function of the maximum degree, independent of the number of vertices of the graph. Finally, we mention that distinguishing labelings of {\em graphs} were also studied by Ahmad et al. \cite{AAB}.

The rest of the paper is structured as follows. In the next section we prove several lemmas that are required for our theorems. In particular, Lemma \ref{l:equivalence} comparing $s^*(n)$ and $s(n)$, Lemma \ref{l:random}, which is the main ingredient in the lower bound on $s(n,m)$ as it implies the existence of a certain (randomly constructed) hypergraph $H$ with large $s(H)$, and Lemma \ref{l:random-sum} about the distribution of the sum of discrete random variables, that we use for the upper bound on $s(n,m)$.
In Section 3 we prove Theorem \ref{t:snm}, our main result.
Section 4 contains the proofs of Theorems \ref{t:Delta}, \ref{t:trees}, and \ref{t:irregular}. The final section consists of concluding remarks and open problems.


\section{Lemmas} 

This section consists of several lemmas facilitating the proof of our main theorems. For a positive integer $x$, we use the notation $[x]=\{1, \ldots, x\}$. 
Our first lemma relates $s^*(n)$, $s(n)$, and $s^*(2n)$.

\begin{lemma}\label{l:equivalence}
	For any $n\geq 2$, we have $s^*(n) \leq  s(n) \leq s^*(2n)$.
\end{lemma}

\begin{proof}
Let $G$ be an $n$-vertex graph with $s^*(G) = s^*(n)$. As mentioned in the introduction, consider a hypergraph $H$ on the vertex set $V=V(G)$ with hypergedges corresponding to the closed neighborhoods of vertices in $G$.
Since a labeling $f$ of $V$ is vertex sum-distinguishing in $G$ if and only if it distinguishing in $H$, we have that $s^*(n) = s^*(G)= s(H) \leq s(n)$ \footnote{Observe that $H$ might have less than $n$ edges since not all closed neighborhoods of $G$ are necessarily distinct, but since adding edges to a hypergraph cannot decrease $s$, we indeed have $s(H) \le s(n)$}.
On the other hand, consider a hypergraph $H$ on a vertex set $B=\{b_1, \ldots, b_n\}$ and with $n$ hyperedges $e_1, \ldots, e_n$, such that $s(H)=s(n)$.
Let $G$  be a graph on vertex set $A\cup B$, where $A= \{a_1, \ldots, a_n\}$, $A\cap B=\emptyset$, 
where $A$ induces a clique with $n$ vertices, $B$ induces an independent set,  and $a_ib_j\in E(G)$ if and only if $b_j\in e_i$. Then we see that if a labeling $f$ is vertex sum-distinguishing in $G$ then, restricted to $B$, it is distinguishing in $H$.
Consider such an optimal $f$, i.e. with a largest label $s^*(G)$. Since $G$ has $2n$ vertices,
$s^*(G) \leq s^*(2n)$.  Thus,  $s(n) = s(H) \leq s^*(G) \leq s^*(2n)$.
\end{proof}

\begin{lemma}\label{l:random}
For any fixed $r\geq 2$, there is a constant $c=c(r)$ such that for every positive integer $N$ it holds that
there exists an $r$-uniform hypergraph $H$ on $N$ vertices such that 
\begin{eqnarray*}
|E(H)|&= & \Theta(N^{(r+1)/2}\sqrt{\log N}), { \mbox and }\\
s(H) &\geq & c N^r.
\end{eqnarray*}
\end{lemma}
\begin{proof}
We are going to extend a result of Bollob\'as and Pikhurko \cite{BP} on distinguishing labelings of graphs to $r$-uniform hypergraphs. Also note that the inequalities in the lemma's statement allow us to assume, whenever necessary, that $N$ is sufficiently large as a function of $r$.

Proof idea:
We provide a lower bound on $s(H)$ for a random $r$-uniform hypergraph $H \sim G_r(N, p)$,
i.e., a hypergraph on a vertex set $[N]$, such that hypergedges are chosen independently with probability $p$. In order to show that $s(H)>s$ for a chosen $s$, we shall consider a fixed labeling $f$ of $[N]$ and denote by $p'$ the probability that $f$ is distinguishing for $H$.
Now, if it holds that $p' = o(s^{-N})$ then we have $\Pr[s(H) \leq s] \leq s^N p' = o(1)$.
So, in this case we see that almost surely $s(H) >s$.

Let $q  = \sqrt{13 r \cdot r!}$, $p = q \sqrt{\ln N}/\sqrt{N^{r-1}}$, $s= \lfloor N^r/(2r\cdot r!) \rfloor$ and $H \sim G_r(N, p)$.
Consider a labeling $f: [N]\rightarrow [s]$. For any $e \in \binom{[N]}{r}$, let $s(e) = \sum_{i\in e}f(i)$. 
We estimate $p'$, the probability that $f$ is distinguishing for $H$.

Let $H_k$ be the $r$-uniform hypergraph on a vertex set $[N]$, with $E(H_k)= \{ e \in  \binom{[N]}{r}\,:\, s(e)=k\}$ and denote $h_k=|E(H_k)|$. 
Note that for any $r$-subset of the vertices  $e \in \binom{[N]}{r}$, $r \le s(e) \leq sr$
and that the $H_k$'s form an edge-decomposition of the complete $r$-uniform hypergraph on the vertex set $[N]$.
Note that $f$ is distinguishing for $H$ if and only if $H$ has at most one edge in each of the $H_k$'s.
We need to consider only those $H_k$'s that have at least two edges so let $K=\{k:  h_k\geq 2\}$. We have 
\begin{eqnarray*}
p'  &= & \Pr[f \mbox{ is distinguishing for } H]\\
     & = & \prod _{k\in K}  \Pr[|E(H)\cap E(H_k)|\leq 1] \\
     & = & \prod _{k\in K}  \left((1-p)^{h_k} + h_kp(1-p)^{h_k-1}\right)\;.\\
\end{eqnarray*}
We need the following statement that is a routine calculation.  If $t_1 \leq t_2 -2  $ then
$$
\left( (1-p)^{t_1} + t_1p(1-p)^{t_1-1} \right) \left((1-p)^{t_2 }+ t_2p(1-p)^{t_2-1}\right)
$$
\begin{equation}\label{e:conv1}  
\leq \left( (1-p)^{t_1+1} + (t_1+1)p(1-p)^{(t_1+1)-1} \right) \left((1-p)^{t_2-1 }+ (t_2-1)p(1-p)^{(t_2-1)-1}\right)\;.
\end{equation}
Using \eqref{e:conv1} we can upper-bound the expression for $p'$ by the one in which each $h_k$ takes an  integer  value $x$ or $x+1$, for some $x$. Let there be $b$ of $x$'s and $|K|-b$  of
$(x+1)$'s, so  $bx+ (|K|-b)(x+1) = \sum_{k\in K} h_k = h$. Assume that $xb\geq h/2$
(the case where $(|K|-b)(x+1) \ge h/2$ is analogous). We have:
\begin{eqnarray*}
p'      & =  & \prod _{k\in K}   (1-p)^{h_k} + h_kp(1-p)^{h_k-1}\\
         & \underset{\eqref{e:conv1}}{\leq} &   \left((1-p)^x + xp(1-p)^{x-1}\right) ^{b} \left((1-p)^{x+1} + (x+1)p(1-p)^{(x+1)-1}\right) ^{|K|-b}\\
         & \leq &  ((1-p)^x + xp(1-p)^{x-1}) ^{b} \\
         & \leq & ((1-p)^x + xp(1-p)^{x-1}) ^{\frac{h}{2}/x}\;.
\end{eqnarray*}         
It is also a routine calculation, that for any any $p$, $0<p<1$
\begin{equation}\label{e:decrease} 
\max_{t\geq 2}  ((1-p)^t + tp(1-p)^{t-1})^{1/t}  \underset{t=2}{=} ((1-p)^2 + 2p(1-p))^{1/2}\;.
\end{equation}
Coming back to bounding $p'$, we have
\begin{eqnarray*}
p'      & \leq  &  ((1-p)^x + xp(1-p)^{x-1}) ^{\frac{h}{2}/x}\\       
         &\underset{\eqref{e:decrease}}{\leq }& ((1-p)^2 + 2p(1-p))^{h/4}\\
         & = & (1-p^2)^{h/4}\\
         & \leq & e^{-p^2h/4}\;.
\end{eqnarray*}
Observe also that the total number of hypergraphs $H_k$ is at most
$$
sr \leq \frac{N^r}{2r!}  \leq  \frac{1}{2}\binom{N}{r} (1+o(1))\;.
$$
Thus, at least about a half of the possible $r$-sets of vertices belong to $H_k$'s that have at least two  edges, i.e., to $H_k$'s, $k\in K$.
In other words, for sufficiently large $N$,
$$
h = \sum_{k\in K} h_k  \geq \binom{N}{r} - \frac{1}{2}\binom{N}{r} (1+o(1)) \geq \frac{N^r}{3r!}\;.
$$
Recall that $p= q \sqrt{\ln N}/\sqrt{N^{r-1}}$ and $q=\sqrt{13r \cdot r!}$. Then
\begin{eqnarray*}
p' & \leq & e^{-p^2h/4}\\
&  \le &  e^{ - q^2(\ln N/ N^{r-1})N^r/(12 r!)}\\
& = & e^{ - (q^2/12r!) N\ln N}\\
& = & e^{ - (13r/12) N\ln N}\\
& \le & N^{-r N}\;.
\end{eqnarray*}
Also recall that  $s= \lfloor  N^r/(2r\cdot r!) \rfloor$ so,
\begin{eqnarray*}
	\Pr[s(H) \leq s] & \le & s^N p'\\
	& \le & s^N N^{-r N}\\
	& = & o(1)\;.
\end{eqnarray*}
This implies that  with high probability $s(H) > s= cN^r$, for  a constant $c$ depending on $r$.  Moreover, with high probability $|E(H)| = \Theta(p\binom{N}{r} )=\Theta(pN^r )= \Theta(N^{(r+1)/2} \sqrt{\log N} )$.
\end{proof}

\vspace{6pt}
Our final lemma of this section upper-bounds the probability that a sum of i.i.d. uniform discrete random variables attains a particular value. We will use it as an ingredient in the upper bound proof of Theorem \ref{t:snm}.

\begin{lemma}\label{l:random-sum} For any constant $C>0$, there exists $\ell_0$ such that for any $\ell > \ell_0$ the following holds. There is an integer $N_0=N_0(\ell)$ such that for any $N>N_0$, if
$X_1, \ldots, X_{2\ell}$ are i.i.d. uniform discrete random variables over $[N]$, 
then for any integer $t$, $\Pr[X_1+\cdots +X_{2\ell} = t] \leq 5/(e^{4C}N)$.
\end{lemma}
\begin{proof}
We first establish the following claim, asserting the concavity of a sum of i.i.d. uniform discrete random variables. 
\begin{claim}\label{claim:0}
	Let $X_1,\ldots,X_\ell$ be i.i.d. uniform discrete random variables over $[N]$ and let
	$X = X_1 + \cdots + X_\ell$. Then for every real $d \ge 0$ it holds that
	$$
	\Pr[ X= (N+1)\ell/2-d] = \Pr[ X= (N+1)\ell/2+d] \ge \Pr[X = (N+1)\ell/2+d+1]\;.
	$$
\end{claim}
\begin{proof}
It will be slightly more convenient to define $W_i=X_i-1$,  $i=1, \ldots, \ell$, $W=W_1+\cdots+W_\ell$, $q=N-1$ and prove 
	the equivalent statement
	$$
	\Pr[W = q\ell/2-d] = \Pr[W = q\ell/2+d] \ge \Pr[W = q\ell/2+d+1]\;.
	$$
	Observe first that $\Pr[W = q\ell/2-d] = \Pr[W = q\ell/2+d]$ as $W$ is symmetric around its mean $q\ell/2$.
	Observe next that $W \in \{0,\ldots,q\ell\}$ so the inequality is only interesting if $d+1 \le q\ell/2$
	and $q\ell/2+d$ is an integer. We prove the claim by induction on $\ell$. The case $\ell=1$ trivially holds as $W_1$ is uniform. Assuming the claim holds for $\ell-1$, we prove it for $\ell$. Let $W^*=W_1+\cdots+W_{\ell-1}$ so
	$W = W^*+W_\ell$. We prove that $\Pr[W = q\ell/2+d] \ge \Pr[W = q\ell/2+d+1]$. Observe that
	$$
		\Pr[W = q\ell/2+d]=\sum_{k=0}^{q} \Pr[W^*=q\ell/2+d-k]\cdot \Pr[W_\ell=k] = \frac{1}{q+1}\sum_{k=0}^{q} \Pr[W^*=q\ell/2+d-k]\;.
	$$
	Similarly,
	$$
		\Pr[W = q\ell/2+d+1]=\frac{1}{q+1}\sum_{k=0}^{q} \Pr[W^*=q\ell/2+d+1-k]\;.
	$$
	So,
	\begin{align*}
		& (q+1)\left(\Pr[W = q\ell/2+d]-\Pr[W = q\ell/2+d+1]\right)\\
		= ~ & \Pr[W^*=q\ell/2+d-q]-\Pr[W^*=q\ell/2+d+1]  \\
		= ~ & \Pr[W^*=q(\ell-1)/2+d-q/2] - \Pr[W^*=q(\ell-1)/2+d+1+q/2]\\
		\ge ~ & 0
	\end{align*}
	where the last inequality follows from the induction hypothesis and from the fact that $|d-q/2| < |d+1+q/2|$.
\end{proof}
To prove the lemma, let 
$$
f_1=X_1+\cdots +X_\ell  \mbox{ ~and~  }  f_2= X_{\ell+1} + \cdots X_{2\ell}\;.
$$
We estimate the probability $\Pr[f_1+f_2=t]$.
We will first need to prove two additional claims. The first is an anti-concentration result for $f_1,f_2$ and the second is a concentration result for them.
Throughout the remainder of the proof we assume that $\ell$ is sufficiently large as a function of $C$ and that $N$ is sufficiently large as a function of $\ell$.
\begin{claim}\label{claim:1}
	Let $j \in \{1,2\}$. For any $C>0$, there exists $\gamma=\gamma(C) > 0$ such that for every real number $x$ it holds that
	\begin{equation}
	\Pr[x-\gamma\sqrt{\ell}N \le f_j \le x+\gamma\sqrt{\ell}N] \le e^{-4C}\;.
	\end{equation}
\end{claim}
\begin{proof}
	Recall that each $X_i$ is uniform discrete over $[N]$. For the sake of our analysis it would be convenient
	to obtain $X_i$ as follows. Let $U_i \sim U[0,1]$ (i.e. $U_i$ is uniform continuous in $[0,1]$).
	Define $X_i = \lceil N U_i \rceil$. Since $\Pr[U_i=0]=0$, we have that $X_i$ is discrete uniform over $[N]$
	as the probability that $X_i=t$ is $1/N$ for each $t \in [N]$.  Denote $g_1=U_1+\cdots+U_\ell$ and $g_2 = U_{\ell+1}+\cdots+U_{2\ell}$ so we have
	$f_j-\ell \le Ng_j \le f_j$ for $j \in \{1,2\}$.
	Since $\ell < \gamma \sqrt{\ell}N$, it suffices to prove that for every real number $y$ it holds that
	$$
	\Pr[ y-2\gamma\sqrt{\ell} \le g_j \le y+2\gamma\sqrt{\ell}] \le e^{-4C}\;.
	$$
	As $g_j$ is an Irwin-Hall distribution (i.e. the sum of i.i.d. copies of $U[0,1]$) with mean $\ell/2$, the maximum of the left hand side is obtained when $y=\ell/2$ so it suffices to prove that
	\begin{equation}\label{e:anti}
	\Pr\left[ \frac{\ell}{2}-2\gamma\sqrt{\ell} \le g_j \le \frac{\ell}{2}+2\gamma\sqrt{\ell}\right] \le e^{-4C}\;.
	\end{equation}
	Since $U_i \sim U[0,1]$, it has mean $1/2$ and standard deviation $1/\sqrt{12}$, so we have by
	the Central Limit Theorem that
	$$
	\lim_{\ell \rightarrow \infty} \Pr\left[ \frac{\ell}{2}-2\gamma\sqrt{\ell} \le g_j \le \frac{\ell}{2}+2\gamma\sqrt{\ell}\right] = \Phi(2\sqrt{12}\gamma)-\Phi(-2\sqrt{12}\gamma)=2\Phi(2\sqrt{12}\gamma)-1\;.
	$$
	Now, choose $\gamma$ such that $2\Phi(2\sqrt{12}\gamma)-1=e^{-4C}/2$. Then  we have
	$$
	\lim_{\ell \rightarrow \infty} \Pr\left[\frac{\ell}{2}-2\gamma\sqrt{\ell} \le g_j \le \frac{\ell}{2}+2\gamma\sqrt{\ell}\right] = \frac{1}{2e^{4C}}\;,
	$$
	implying that for all $\ell$ sufficiently large as a function of $C$ we have that \eqref{e:anti} holds.
\end{proof}
	
\begin{claim}\label{claim:2}
	Let $j \in \{1,2\}$. It holds that
	\begin{equation}
		\Pr[|f_j- \ell(N+1)/2|  \ge \ell^{2/3} N] \le \frac{1}{\ell}\;.
	\end{equation}
\end{claim}
\begin{proof}
	As in the proof of the previous claim, since $f_j-\ell \le N g_j \le f_j$ and since $2\ell \le \ell^{2/3}N$,
	it suffices to prove that
	$$
		\Pr\left[|g_j-\ell/2| \ge \frac{1}{2}\ell^{2/3}\right] \le \frac{1}{\ell}\;.
	$$
	Since $g_j$ is the sum of $\ell$ i.i.d. random variables, each in $[0,1]$, and each with mean $\frac{1}{2}$, it follows by Chernoff's inequality (see, e.g. \cite{AS}, Appendix A) that
	$$
	\Pr\left[|g_i-\ell/2| \ge \frac{1}{2}\ell^{2/3}\right] \le 2e^{-\ell^{4/3}/(8\ell)} =2e^{-\ell^{1/3}/8} \le \frac{1}{\ell}\;.
	$$
\end{proof}

Armed with the three claims we proceed as follows. Since $f_1$ and $f_2$ are independent 
and since $\ell \le f_j \le \ell N$ we have that, for any $t$, $0\leq t \leq 2 \ell N$, 
$$
\Pr[f_1+f_2=t] = \sum_{k=\ell}^{\ell N} \Pr[f_1=k]\cdot \Pr[f_2=t-k]\;.
$$
We cover $\{\ell,\ldots,\ell N\}$ with five (not necessarily disjoint) sets $S_1,S_2,S_3,S_4,S_5$ defined as follows.
\begin{align*}
S_1 = & \{k\,|\, \ell(N+1)/2 - \gamma \sqrt{\ell} N  \le k \le \ell(N+1)/2 + \gamma \sqrt{\ell} N\}\,\\
S_2 = & \{k\,|\, \ell(N+1)/2 - \gamma \sqrt{\ell} N  \le t-k \le \ell(N+1)/2 + \gamma \sqrt{\ell} N\}\,\\
S_3 = & \{k\,|\, |k- \ell(N+1)/2|\ge \ell^{2/3} N\}\,\\
S_4 = & \{k\,|\, |(t-k) - \ell(N+1)/2|\ge \ell^{2/3} N\}\,\\
S_5 = & \{\ell,\ldots,\ell N\} \setminus (S_1 \cup S_2 \cup S_3 \cup S_4)\;.
\end{align*}
For $z \in \{1,2,3,4,5\}$ let $J_z = \sum_{k \in S_z} \Pr[f_1=k]\cdot \Pr[f_2=t-k]$ so that we have
$$
\Pr[f_1+f_2=t] \le J_1+J_2+J_3+J_4+J_5\;.
$$
We now bound each $J_z$ where we will use Claim \ref{claim:0}, Claim \ref{claim:1}, Claim \ref{claim:2},
and the trivial bound $\Pr[f_j=k'] \le 1/N$ which holds for every $k' \in [N]$ since
$f_j$ is the sum of discrete random variables, each uniform on $N$ possible values.
By the definition of $S_1$ and by Claim \ref{claim:1} applied to $f_1$ with $x=\ell(N+1)/2$:
\begin{equation*}
	J_1 = \sum_{k \in S_1} \Pr[f_1=k]\cdot \Pr[f_2=t-k] \le \frac{1}{N}\sum_{k \in S_1} \Pr[f_1=k] \le \frac{1}{Ne^{4C}}\;.
\end{equation*}
Similarly, by the definition of $S_2$ and by Claim \ref{claim:1} applied to $f_2$ with $x=\ell(N+1)/2$:
\begin{equation*}
	J_2 = \sum_{k \in S_2} \Pr[f_1=k]\cdot \Pr[f_2=t-k] \le \frac{1}{N}\sum_{k \in S_2} \Pr[f_2=t-k] \le \frac{1}{Ne^{4C}}\;.
\end{equation*}
By the definition of $S_3$ and by Claim \ref{claim:2} applied to $f_1$:
\begin{equation*}
	J_3 = \sum_{k \in S_3} \Pr[f_1=k]\cdot \Pr[f_2=t-k] \le \frac{1}{N}\sum_{k \in S_3} \Pr[f_1=k] \le \frac{1}{N\ell} \le \frac{1}{Ne^{4C}}\;.
\end{equation*}
By the definition of $S_4$ and by Claim \ref{claim:2} applied to $f_2$:
\begin{equation*}
	J_4 = \sum_{k \in S_4} \Pr[f_1=k]\cdot \Pr[f_2=t-k] \le \frac{1}{N}\sum_{k \in S_4} \Pr[f_2=t-k] \le \frac{1}{N \ell} \le \frac{1}{Ne^{4C}}\;.
\end{equation*}

Finally consider $J_5$. To estimate it, we will distinguish between two cases, according to the value of $t$.
Assume first that $t \le 2\ell(N+1)/3$ or $t \ge 4\ell(N+1)/3$.
In this case $S_3 \cup S_4 = \{\ell,\ldots,\ell N\}$ and hence $S_5=\emptyset$
implying that $J_5=0$.  Assume next that $2\ell(N+1)/3 < t < 4\ell(N+1)/3$.
First, observe that the number of elements of $S_5$ is at most $2\ell^{2/3}N+1 < 3\ell^{2/3}N$ as it is disjoint from, say, $S_3$. Consider some term of $J_5$, namely $\Pr[f_1=k]\cdot \Pr[f_2=t-k]$ where
$k \in S_5$. By Claim \ref{claim:0}, we have that $\Pr[f_1=k] \le \Pr[f_1 = k^*]$ where $k^* \in S_1$
as $k^*$ is closer to the mean $\ell(N+1)/2$ than $k$ is. But the number of elements in $S_1$ is at least
$2\gamma\sqrt{\ell}N$ so we must have $\Pr[f_1=k] \le 1/{|S_1|} \le 1/(2\gamma\sqrt{\ell}N)$.
Similarly, by Claim \ref{claim:0}, we have that $\Pr[f_2=t-k] \le \Pr[f_2 = t-k^*]$ where $k^* \in S_2$
as $t-k^*$ is closer to the mean $\ell(N+1)/2$ than $t-k$ is.
But the number of elements in $S_2$ is at least
$2\gamma\sqrt{\ell}N$ so we must have $\Pr[f_2=t-k] \le 1/{|S_2|} \le 1/(2\gamma\sqrt{\ell}N)$.
Hence, in any case,
\begin{equation*}
	J_5 = \sum_{k \in S_5} \Pr[f_1=k]\cdot \Pr[f_2=t-k] \le 3\ell^{2/3}N \cdot \left(\frac{1}{2\gamma\sqrt{\ell}N}\right)^2\ \le \frac{1}{Ne^{4C}}\;.
\end{equation*}
We have thus proved that $\Pr[f_1+f_2=t] \le 5e^{-4C}/N$, as required.
\end{proof}


\section{Proof of the main result}

\subsection{Proof of the lower bound of Theorem \ref{t:snm}}
Let $\epsilon$ be given, $0<\epsilon <1$.  Let $n$  be sufficiently large and $m$ be given such that $n\leq m \leq n^{1/\epsilon}$. 
We shall construct a hypergraph $H$ on $n$ vertices and  $m$ hyperedges such that $s(H) \geq m^{2-\epsilon}$.  Let $r$ be a positive integer such that $\epsilon > 2/(r+1)$.
Recall that Lemma \ref{l:random} implies, for sufficiently large $N$ and any positive $\delta$, the existence of a hypergraph
$H'$ on $N$ vertices and  $N^{(r+1)/2 +\delta}$ hyperedges satisfying
$s(H') \ge cN^r$, for a constant $c=c(r)$.\footnote{We ignore rounding issues as these have no effect on the asymptotic statement of the theorem.}
Note that  $N^{(r+1)/2 +\delta}$  is slightly larger than the expression for the number of hyperedges  given in Lemma \ref{l:random}, but we can always add hyperedges if necessary as this does not decrease the parameter $s$.
Next, we choose $N$ such that $m =  N^{(r+1)/2 + \delta}$.  Thus  $H'$ has  $m$ hyperedges.
Note that:
$$
N \le (N^{(r+1)/2 + \delta})^{2/(r+1)} = m^{2/(r+1)} \le (n^{1/\epsilon})^{2/(r+1)} \le n,
$$
so by just adding $n-N$ isolated vertices to $H'$   we obtain a hypergraph $H$ with $n$ vertices and  $m$ hyperedges
and with $s(H) = s(H') \ge cN^r$. 
Hence for $\delta$ sufficiently small
$$
s(H) \ge cN^r \ge c (m^{2/(r+1+2\delta)})^r \ge m^{2-\epsilon}\;.
$$
\qed

\subsection{Proof of the upper bound of Theorem \ref{t:snm}}

Consider a hypergraph $H=(V,E)$ on $n$ vertices and $m$ hyperedges.  We shall argue that an appropriate random labeling is distinguishing with positive probability. 
Before we prove our upper bound $m^2/C$ on $s(H)$, we shall quickly remark that the upper bound $s(H) \leq m^2$ is easy to obtain. Indeed, to each vertex assign an integer value from $[m^2]$   independently with probability $1/m^2$. Consider the probability $p$ that two given distinct hyperedges $e$ and $e'$  get the same sum of the labels. Fix an arbitrary vertex $y$  in the symmetric difference of $e$ and $e'$. Then assuming that all other labels in the union of $e$ and $e'$ are fixed, there is at most one value of the label assigned to $y$ that makes the sum of labels in $e$ and $e'$ the same. Thus $p\leq 1/m^2$.  Taking the union bound over all $\binom{m}{2}$ pairs of hyperedges, we see that the probability that the labeling is not distinguishing is at most $\binom{m}{2}/m^2<1$. \\

Next we shall improve this easy upper bound to $s(n,m)=o(m^2)$. This turns out to require significantly more effort. We first describe the main idea of the proof.
We consider a hypergraph $H=(V,E)$ on $n$ vertices and $m$ hyperedges. Let $C>0$ and  $N= \lceil m^2/C \rceil$.  Consider  a labeling $f:V(H)\rightarrow [N]$ such that $f(v)$ is assigned randomly with $\Pr[f(v)=i]= 1/N$ for any $i\in [N]$ and assignment of values to distinct vertices is independent. Let, for any set $Q$ of vertices, $s(Q)$ denote $\sum_{v\in Q}f(v)$. For two hyperedges $e,e'$, let $X(e, e') = e\setminus e'$. Observe that a vertex labeling $f$ is distinguishing on $H$ if for any two hyperedges $s(X(e,e'))\neq s(X(e', e))$. Let $B(e,e')$ be the (bad) event that $s(X(e,e'))= s(X(e', e))$.

Consider  sets $D(e,e') = X(e,e')\cup X(e',e)$ and split the analysis into cases depending on the size of $D(e,e')$. For small $D(e,e')$  we would like to apply the Lov\'asz Local Lemma, but of course the lemma's dependency digraph might have a high degree if there are vertices that belong to many such $D(e,e')$’s, called ``dangerously popular'' vertices. We treat them first observing that there are not so many such vertices. Finally, we deal with large $D(e,e')$’s. For those we show that the bad event $B(e,e')$ does not happen by choosing a large set $S$ of size $2\ell$  in $X(e,e')$ or in $X(e',e)$, fixing the labels on the remaining vertices in $D(e,e')$ and showing that
$\Pr[B(e,e')] \leq  \Pr[s(S) =t]$ for a specific value $t$, finally upper-bounding the latter using Lemma \ref{l:random-sum}. We now proceed with the detailed proof.

\vspace{5pt} For our fixed $C$,  let $K > P > C$ where $K$ and $P$ are positive integer constants chosen to satisfy the claimed inequalities used in the proof. They will only depend on $C$. For the rest of the proof we assume that $C>3$ and note that if the theorem holds for some value of $C$, it holds for any smaller positive value of $C$.

\begin{itemize}
\item
A pair of hyperedges $e, e'$   is {\em dangerous} if $|D(e,e')| \le K$. Otherwise, the pair is called {\em non-dangerous}.
\item
We call a vertex $w \in V(H)$ {\em dangerously popular} if for at least $m^2/K^3$ dangerous pairs $e,e'$ it holds that $w \in D(e,e')$.  Let $S$ be the set of all dangerously popular vertices. 
\item
For a pair $e,e' \in E(H)$ (whether dangerous or not) let $Y(e,e')= X(e, e')\cap S$, the set of dangerously popular vertices in $X(e,e')$ and let $Z(e,e')=X(e,e') \setminus Y(e,e')$.
\item
We call a pair $e,e' \in E(H)$ {\em special} if each vertex of $D(e,e')$ is dangerously popular, i.e. $D(e,e')=Y(e,e')\cup Y(e',e)$.
\item
Two special pairs $e_1,e'_1$ and $e_2,e'_2$ are {\em equivalent} if $\{X(e_1,e'_1),X(e'_1,e_1)\}=\{X(e_2,e'_2),X(e'_2,e_2)\}$. Observe that ``equivalent'' is an equivalence relation over the special pairs.
\item
We call a non-dangerous and non-special pair $e,e' \in E(H)$ {\em newly dangerous} if all but at most $P$ vertices of $D(e,e')$ are dangerously popular (so $1 \le |Z(e,e')\cup Z(e',e)| \le P$ for such pairs).
\end{itemize}

We observe that that the number of dangerously popular vertices is  $|S| \le K^4$.
Indeed, the total sum of cardinalities of all the $D(e,e')$'s  ranging over all dangerous pairs is at most
$K\binom{m}{2}$ and as each dangerously popular vertex is counted at least $m^2/K^3$ times, there are at most
$K\binom{m}{2}/(m^2/K^3) \le K^4$ dangerously popular vertices.

Recall that $N=\lceil m^2/C \rceil$. Our assignment of values from $[N]$ to the vertices of $H$ proceeds in two steps.
We will first assign values to the dangerously popular vertices such that some properties are guaranteed.
We will then assign values to the remaining vertices.

\vspace*{6pt}\noindent 
{\bf Step 1:} Assign random values to the dangerously popular vertices (i.e. the vertices in $S$).
As in the proof of Lemma \ref{l:random-sum}, for the purpose of our analysis, the random values are assigned as follows.
Each $w \in S$ is assigned uniformly and independently a random {\em real} $g(w)$
in $[0,N]$. Then, we define $f(w)=\lceil g(w) \rceil$. Since $\Pr[f(w)=0]=0$, we have that $f(w)$ is discrete uniform in $[N]$
as the probability that $f(w)=t$ is $1/N$ for each $t \in [N]$.

\vspace*{6pt}\noindent
Recall that $Y(e,e') = X(e,e') \cap S$. Let $f(e,e') = \sum_{w \in Y(e,e')} f(w)$. We say that Step 1 is {\em successful} if both of the following hold:

1. For every special pair $e,e'$ we have $f(e,e') \neq f(e',e)$.

2. For at most $m^2e^{-4C}$ newly dangerous pairs $e,e'$ it holds that $|f(e,e')-f(e',e)| \le PN$.

 \begin{lemma}\label{l:successful}
 With positive probability, Step 1 is successful.
 \end{lemma}
 Lemma \ref{l:successful} will be proved later, but for now assume that it holds,
 so fix an assignment of the vertices of $S$ such that Step 1 is successful.

\vspace*{6pt}\noindent
{\bf Step 2:} Assign random values to the remaining $n-|S|$ vertices.
As in Step 1, we assign the random values are follows.
Each $w \in V(H) \setminus S$ is assigned uniformly and independently a random {\em real} $g(w)$
in $[0,N]$. Then, we define $f(w)=\lceil g(w) \rceil$. Recall that $f(w)$ is discrete uniform in $[N]$.
This now defines for each hyperedge $e \in E(H)$ the sum $s(e)=\sum_{w \in e} f(w)$. 
We need to estimate the probability that $s(e)=s(e')$ for distinct hyperedges $e, e'$.
We partition the  pairs $(e, e')$ of hyperedges  into five types:
\begin{enumerate}[(a)]
\item  The special pairs.
\item    The newly dangerous pairs for which  $|f(e,e')-f(e',e)| > PN$.
\item  The newly dangerous pairs for which $|f(e,e')-f(e',e)| \le PN$.
\item   Non-dangerous pairs that are not newly dangerous and not special.
\item  Dangerous pairs that are not special.
\end{enumerate}

\vspace{6pt}\noindent
We refer to these types by their letter. Each pair of hyperedges is  of precisely one of these types. We now analyze each type.
Let $A_a, A_b, A_c, A_d, $ and $A_e$ be events that there is a pair $e, e'$ of type (a), (b), (c), (d), or (e), respectively, such that $s(e)=s(e')$.
We prove the following lemmas later.
\begin{lemma}\label{l:type-a-b}
$\Pr[A_a]=\Pr[A_b]=0$.
\end{lemma}
\begin{lemma}\label{l:type-c}
$\Pr[A_c] \le e^{-3C}$.
\end{lemma}
\begin{lemma}\label{l:type-d}
$\Pr[A_d] \le e^{-3C}$.
\end{lemma}
\begin{lemma}\label{l:type-e}
$\Pr[A_e] \le 1-e^{-2C}$.
\end{lemma}

Lemmas \ref{l:type-a-b}, ~\ref{l:type-c},~\ref{l:type-d}, and \ref{l:type-e}
imply  that $\Pr [A_a \cup A_b \cup A_c \cup A_d \cup A_e] \leq e^{-3C}+e^{-3C}+1-e^{-2C} < 1$. 
Thus, with positive probability none of these bad events happen and there is a desired distinguishing labeling of $H$.
It remains to prove Lemmas  \ref{l:successful},  \ref{l:type-a-b},  \ref{l:type-c}, \ref{l:type-d}, and \ref{l:type-e}.\\

In several proofs we shall need the following observation for any distinct subsets $X$ and $X'$ of vertices,  recalling that $s(X)=\sum_{w\in X} f(w)$,
\begin{equation}\label{star} 
\Pr(s(X)=s(X'))\leq 1/N.
\end{equation}
The reason for this observation to hold is the same as we outlined in the first paragraph of the proof -  fixing all labels except for one vertex, say  $y$,  in the symmetric difference of $X$ and $X'$, we see that $\Pr(s(X)=s(X'))\leq \Pr(f(y)=t)=1/N$, for some specific value $t$.

\begin{proof}[Proof of Lemma \ref{l:type-a-b}]
If $e,e'$ is a  pair of type (a), then clearly $s(e)-s(e')=f(e,e')-f(e',e)$. But since Step 1 is successful, we have that $f(e,e') \neq f(e',e)$ and hence $s(e) \neq s(e')$. Thus the event $A_a$ never happens.

If $e, e'$ is  a pair of type (b), i.e., a newly-dangerous pair for which $ |f(e,e')-f(e',e)| > PN$  we proceed as follows.
Assume without loss of generality that  $f(e,e')-f(e',e) > PN$.
Clearly
$$
s(e) = f(e,e')+\sum_{w \in e \cap e'} f(w) + \sum_{w \in Z(e,e')} f(w) \geq f(e,e')+\sum_{w \in e \cap e'} f(w)\;.
$$
On the other hand,
$$
s(e') = f(e',e)+\sum_{w \in e \cap e'} f(w) + \sum_{w \in Z(e',e)} f(w) \leq  f(e',e)+\sum_{w \in e \cap e'} f(w) + PN,
$$
because $|Z(e',e)| \le P$ by the definition of newly-dangerous.
It follows from the last two inequalities that
$$
s(e)-s(e') \ge  f(e,e') - f(e',e) - PN > 0,
$$
so we have that $s(e) \neq s(e')$. Thus the event $A_b$ never happens.
\end{proof}

\begin{proof}[Proof of Lemma \ref{l:type-c}]
Let $e,e'$ be a pair of type (c), namely a newly dangerous pair for which it holds that $|f(e,e')-f(e',e)| \le PN$. As Step 1 is successful, we have that the number of pairs of type (c) is at most $m^2e^{-4C}$.

By (\ref{star}), we have that $\Pr[s(e)=s(e')] \le 1/N$. Since the number of pairs of type (c) is at most $m^2e^{-4C}$ we have that
$$
\Pr[A_c] \le \frac{m^2e^{-4C}}{N} = \frac{m^2e^{-4C}}{\lceil m^2/C \rceil}  \le Ce^{-4C} \le e^{-3C}.
$$
\end{proof}

\begin{proof}[Proof of Lemma \ref{l:type-e}]
For a pair $e,e'$ of type (e), let $A(e, e')$ be the event that $s(e)=s(e')$. Using (\ref{star}) we have $\Pr[A(e,e')] \le 1/N$.
Letting $L$ denote the set of pairs of type (e), our goal is to prove that $\Pr[\cap_{\{e,e'\} \in L} \overline{A(e,e')}] \ge e^{-2C}$ as this is equivalent to proving that $\Pr[A_{e}] \le 1-e^{-2C}$.
To this end, we will use the Lov\'asz Local Lemma (LLL).
Consider the dependency digraph on the events $A(e,e')$ (note: there could be as many as $\binom{m}{2}$ such events).
We claim that any event $A(e,e')$ depends on not too many other events.
Indeed, if $Z(e_1,e_1') \cup Z(e_1',e_1)$ is disjoint from $Z(e_2,e_2') \cup Z(e_2',e_2)$, then the event $A(e_1,e_1')$ is independent of the event $A(e_2,e_2')$ as they involve assignment of values to disjoint sets of vertices.
Recall that the pairs of  type (e) are, in particular, dangerous pairs.
Hence $|Z(e,e') \cup Z(e',e)| \le K$, for any pair $e, e'$ of type (e). Furthermore, each vertex of $Z(e,e') \cup Z(e',e)$ is not dangerously popular.
Thus, we have that $A(e,e')$ is independent of all but at most $K \cdot m^2/K^3=m^2/K^2$ other events.  Denote $\{e_1,e_1'\} \sim \{e_2,e_2'\}$ if $Z(e_1,e_1') \cup Z(e_1',e_1)$ is not disjoint from $Z(e_2,e_2') \cup Z(e_2',e_2)$. To apply LLL, define $x(e,e')=2/N$. For any $e_1,e_2$ of type (e) it now holds that
$$
x(e_1,e_2)\Pi_{\{e_1',e_2'\} \sim \{e_1,e_2\}}(1-x(e_1',e_2')) \ge \frac{2}{N}\left(1-\frac{2}{N}\right)^{m^2/K^2}
$$
$$
\ge 
\frac{2}{N}\left(1-\frac{2}{N}\right)^{CN/K^2} > \frac{1}{N} \ge \Pr[A(e_1,e_2)]
$$
so the condition in the statement of LLL holds. So, by the LLL, we have that
$$
\Pr[\cap_{\{e,e'\} \in L} \overline{A(e,e')}] \ge (1-x(e,e'))^{|L|} \ge \left(1-\frac{2}{N}\right)^{m^2/2} \ge e^{-2C},
$$
as required. 
\end{proof}

\begin{proof}[Proof of Lemma \ref{l:successful}]
We first prove that with probability at least $2/3$, for every special pair $e,e'$ we have $f(e,e') \neq f(e',e)$.
Observe that the number of equivalence classes in the ``equivalent'' relation is at most
$2^{|S|}2^{|S|} \le 4^{K^4}$ (namely, a constant).
Since for two equivalent special pairs $ e_1,e'_1$ and $e_2,e'_2$ we have that
$f(e_1,e_1') \neq f(e_1',e_1)$ if and only if $f(e_2,e_2') \neq f(e_2',e_2)$,
it suffices to consider a representative special pair from every equivalence class. Now, if $e,e'$ is a special pair then,
using (\ref{star})
we have that $\Pr[f(e,e')=f(e',e)] \le 1/N$. We have by the union bound that the probability that for some special pair $f(e,e')=f(e',e)$ is at most $4^{K^4}/N \ll 1/3$.
So, with probability at least $2/3$, for every special pair $e,e'$ we have $f(e,e') \neq f(e',e)$.

We next prove that with probability at least $2/3$, for at most $m^2e^{-4C}$ newly dangerous pairs it holds that $|f(e,e')-f(e',e)| \le PN$
(thus, we will have that Step 1 is successful with probability at least $1-(1-2/3)-(1-2/3) > 0$, as required).
To prove this we will need to establish some ``anti-concentration'' result, and this will be possible by applying the law of large numbers to some appropriate random variable.

Let us fix a newly dangerous pair $u,v$. We know that $u,v$ is not a dangerous pair, namely $|D(e,e')| \ge K$.
On the other hand, we know that $D(e,e')$ contains many dangerously
popular vertices, since $1 \le |Z(e,e')| \le P$. So, either $|Y(e, e')| \ge (K-P)/2 \ge K/4$ or else $|Y(e',e)| \ge (K-P)/2 \ge K/4$.
Assume without loss of generality that $|Y(e,e')| \ge K/4$. Now, suppose we are given that $f(e',e)=t$ for some integer $t$. Given this information, we would like to upper bound the probability that $f(e,e')$ lies in $[t-PN, t+PN]$. If we can provide an upper bound which does not depend on $t$, then we have upper-bounded the probability that $|f(e,e')-f(e',e)| \le PN$ regardless of any given information.

So, consider indeed the random variable $f(e,e')$.
It is the sum of $\ell=|Y(e,e')| \ge K/4$ i.i.d. random variables, namely $f(e,e')=X_1+\cdots+X_\ell$ where each $X_i$ is discrete uniform in $[N]$. It will be slightly more convenient to normalize as follows.
Recall that each $X_i$ corresponds to some $f(w)$ for $w \in Y(e,e')$ and that $f(w)=\lceil g(w) \rceil$.
Hence $X_i$ is determined by first selecting uniformly at random a real number $W_i$ in $[0,N]$
and then setting $X_i = \lceil W_i \rceil$. Define $U_i=W_i/N$ and notice that
$U_i \sim U[0,1]$ and that $X_i=\lceil NU_i \rceil$.

Let $g(e,e')=U_1+\cdots+U_\ell$ and observe that $f(e,e')-\ell \le N g(e,e') \le f(e,e')$.
Thus, it suffices to upper bound the probability that $g(e,e')$ lies in $[t/N-P-\ell/N, t/N+P]$.
As $\ell/N \le K/N \le P$, it suffices to upper bound the probability that $g(e,e')$ lies in $[t^*-2P, t^*+2P]$
for some real number $t^*$. As $U_1+\cdots+U_\ell$ is an Irwin–Hall distribution which is concave in $[0,\ell]$,
the latter probability is maximized when $t^*=\ell/2$, so it remains to upper bound the probability that
$g(e,e')$ lies in $[\ell/2 - 2P, \ell/2+2P]$. As the $U_i$ are i.i.d. each having mean $\frac{1}{2}$ and standard deviation $1/\sqrt{12}$ (i.e. absolutely bounded standard deviation), the (weak) law of large numbers applies to their sum $g(e,e')$, namely for every constant $P$
$$
\lim_{\ell \rightarrow \infty} \Pr[g(e,e') \in [\ell/2 - 2P, \ell/2+2P]] = 0\;.
$$
This, in turn, means that for all $K$ sufficiently large as a function of $P,C$
(hence all $\ell$ sufficiently large since $\ell \ge |K|/4)$, 
$$
\Pr[g(e,e') \in [\ell/2 - 2P, \ell/2+2P]] \le \frac{1}{3}e^{-4C}\;.
$$
We have thus proved that $\Pr[|f(e,e') - f(e',e)| \le PN] \le e^{-4C}/3$.
As there are less than $m^2$ pairs to consider, we have that the expected number of newly dangerous pairs satisfying $ |f(e,e')-f(e',e)| \le PN$ is at most
$m^2e^{-4C}/3$. By Markov's inequality the probability that there are more than $m^2e^{-4C}$ such pairs is less than $1/3$,
so indeed with probability at least $2/3$, for at most $m^2e^{-4C}$ newly dangerous pairs it holds that $|f(e,e')-f(e',e)| \le PN$.
\end{proof}

\begin{proof}[Proof of Lemma \ref{l:type-d}]
Let $e,e'$ be a pair of type (d), namely it is a non-dangerous pair and is not newly dangerous nor special.
We will prove that $\Pr[s(e)=s(e')] \le e^{-3C}/m^2$.
The lemma then follows as there are less than $m^2$ such pairs to consider.
Being non-dangerous and not newly dangerous means that $|Z(e,e') \cup Z(e',e)| \ge P$.
Assume without loss of generality that $|Z(e,e')| \ge P/2$.
Let $\ell=\lfloor P/4 \rfloor$ and let 
$Z$ be a subset of $Z(e, e')$ of size $2\ell$.

Suppose we are given the value of $f(w)$ for all $w \in V(H) \setminus Z$.
Then, conditioned on this information, for $s(e)=s(e')$ to hold, $s(Z)$ must avoid a particular value $t$.
By Lemma \ref{l:random-sum} we have that $\Pr[s(Z)=t]\leq 5/(e^{4C}N)$.
Thus using the union bound over all pairs of hyperedges of type (d), we have that 
$$
\Pr[A_d]\leq \frac{m^2}{2} \frac{5}{e^{4C}N} \leq \frac{5C}{e^{4C}} \le e^{-3C}\;.
$$
\end{proof}


\section{Proofs of Theorems \ref{t:Delta}, \ref{t:trees}, \ref{t:irregular}.}

\subsection{Proof of Theorem \ref{t:Delta}}

Let $G$ be a nonempty $n$-vertex graph with  maximum degree $\Delta$, minimum degree $\delta$, and the largest number of  vertices with pairwise distinct closed neighbourhoods equal to  $n'$.
We aim to show that
$$
\frac{n'+ \delta}{\Delta +1} \leq s^*(G) \leq  \max \{(n-d(v)-1)(d(v)+1)+2\,:\,v\in V(G)\}\;.
$$
For a vertex labeling $f$, we say that a pair of vertices $u, v$ is {\it bad} if $N[u]\neq N[v]$ and $s^*_f(u)=s^*_f(v)$, otherwise the pair is {\it good}. Thus, a labeling is vertex sum-distinguishing if all pairs are good.
Let $\xi =  \max \{(n-d(v)-1)(d(v)+1)+2 : ~ v\in V(G)\}$, where $d(v)$ is the degree of vertex $v$. Consider a labeling $f:V(G) \rightarrow [\xi]$ with a smallest number of bad pairs. We argue that the number of bad pairs is, in fact, zero.

If not, let $u,v$ be a bad pair. Let $x=u$ if $u$ and $v$ are not adjacent and otherwise let $x=w$, for some $w\in (N(v)\setminus N(u))\cup (N(u)\setminus N(v))$. Note that changing the label for $x$ makes the pair $u,v$  good.
We shall change the label of $x$ such that no good pair becomes bad, i.e., so that the number of bad pairs decreases. Denote the new labeling $f'$.  Let $t$ be a new value assigned to $x$, i.e., $t\neq f(x)$, 
$f'(x) =t$, $f'(z)=f(z)$,  for any $z\in V(G)-x$.

We see that $s^*_{f'}(y)=s^*_f(y)$ if $y\not\in N[x]$ and $s^*_{f'}(y)=s^*_f(y)-f(x)+t$ if $y\in N[x]$.
Thus $s^*_{f'}(y) \neq s^*_{f'}(y')$ if $s^*_f(y)\neq s^*_f(y')$ and ($y, y' \in N[x]$ or $y,y'\in V(G) -N[x])$. 
We have that $s^*_{f'}(y)=s^*_{f'}(y')$ for $y\in N[x]$ and $y'\not\in N[x]$ if and only if $s^*_f(y) -f(x)+t= s^*_f(y')$. So, a new bad pair can only appear if one vertex is in $N[x]$ and another is not.

Choose $t \in Q$, where
$$
Q= [\xi] \setminus \left( \{f(x)\}\cup  \{ s^*_f(y')-s^*_f(y)+f(x): ~y\in N[x], y'\not\in N[x]\}\right).
$$
Since 
\begin{eqnarray*}
|\{ s^*_f(y')-s^*_f(y)+f(x): y\in N[x], y'\not\in N[x]\}|&\leq &|\{ (y,y'): y\in N[x], y'\not\in N[x]\}|\\
& = &(n-d(x)-1)(d(x)+1)\\
&\leq & \xi-2,
\end{eqnarray*}
 the set $Q$ is non-empty, so there is a choice of $t\leq \xi$,  such that  $t\neq f(x)$ and  $s^*_{f'}(y)\neq s^*_{f'}(y')$ for any $y\in N[x]$ and $y'\not\in N[x]$.
Since there is no bad pair $y, y'$ for $y, y'\in N[x]$ or $y, y'\not\in N[x]$  in $f'$  that was not bad in $f$
and the pair $u,v$ that was bad in $f$ is no longer bad in $f'$, we see that the number of bad pairs in $f'$ is strictly less than the number of bad pairs in $f$, a contradiction.

For the lower bound, observe that if $f$ is a vertex sum-distinguishing labeling of $G$ with the largest label $k$, then $S= \{s^*_f(v): v\in V(G)\}\subseteq \{ (\delta(G)+1)\cdot 1, \ldots, (\Delta+1)\cdot k\}$. Since $|S|\geq n'$, we have $n' \leq |S|\leq (\Delta+1)k- (\delta +1) +1$, giving the desired lower bound.
 
Note that the lower bound is tight for any  pair $\delta, \Delta$, $\delta\leq \Delta$. If $\delta =\Delta$ consider $G=K_{\Delta+1}$,  for which $n'=1$, $s^*(G)=1$, and $(n'+\delta)/(\Delta+1)=1$.
If $\delta<\Delta$, consider $G$ that is a vertex-disjoint union of $K_{\Delta +1}$ and $K_{\delta +1}$. In this case $n'=2$, $s^*(G)=1$, and $\lceil (n'+\delta)/(\Delta+1)\rceil =1$.
\qed
 
\subsection{Proof of Theorem \ref{t:trees}}

For a tree $T$ and its vertex $u$,  let  $L(u)$ be the set of leaves adjacent to $u$. Let  $L(T)  = \max \{ |L(u)|\,:\,u \in V(T)\}$.
We shall prove for $n \geq 3$ and any  tree $T$  on $n$ vertices, that  $s^*(T) \leq   2n-2-L(T)$ by induction on $n$. Note that this bound is sharp for stars.

The case $n=3$ holds vacuously since $T$ is a star.    
Suppose the statement holds for $n\geq 3$, and suppose $T$ has $n+1$ vertices and is not a star. Choose a vertex $u$ for which  $L(T) = |L(u)|$,  choose a leaf  $v$ adjacent to $u$, and let $T^*=T-v$.  Then   $L(T)-1 = |L(u) |-1  \leq  L(T^*)$. By induction there is a vertex sum-distinguishing labeling $f : V(T^*) \rightarrow [2n -2-L(T^*)]$ of  $T^*$. 
Observe that
$$
2n- 2- L(T^*) \leq 2n - 2 - (L(T) -1) = 2n -1 - L(T) < 2n - L(T) = 2(n+1) -2 - L(T)\;.
$$
We define a labeling $f' : V(T) \rightarrow [2n-L(T)]$ such that $f'(u) = f(u)$, $u\in V(T^*)$, $f'(v)=\xi$. 

We argue that we can find an appropriate $\xi$ so that the labeling $f'$ does not contain bad pairs, i.e., pairs of vertices $y, y'$ such that $N[y]\neq N[y']$ but $s^*_{f'}(y) = s^*_{f'}(y')$. Since there are no bad pairs in $T^*$ under $f$, we have that $y, y'$ is not a bad pair if $y, y'\in T-\{u, v\}$.
Thus we need to consider only the pairs $y, y'$, where $y\in \{u, v\}$. Let $L= L(u)$ in $T$.

If $y=u$ and $y'\in L$, we see that $s^*_{f'}(u)> s^*_{f'}(y')$ regardless of $\xi$. Thus such a pair $y, y'$ is not bad.
A pair $y=u, y'=u'$, $u'\in V(T^*)-(\{u\}\cup L)$ can be bad if $s^*_{f'}(u) = s^*_f(u)+\xi = s^*_{f}(u')$.
A pair $y=v, y'=u'$, $u'\in V(T^*)-u$ can be bad if $s^*_{f'}(v) = f(u)+\xi = s^*_{f}(u')$.
Thus, if $\xi \not\in X$, where
$$
X= \{ s^*_{f}(u')-s^*_f(u): u'\in V(T^*)-(\{u\} \cup L)\} \cup \{ s^*_{f}(u')-f(u): u'\in V(T^*)-u\}\;,
$$
then $f'$ has no bad pairs on $T$. Note that $|X| \leq  (n-1) -|L| + (n-1) = 2n -|L| -2$. 
Thus, there is an available choice for $\xi$ in $[2n-|L|]-X$.
\qed

\subsection{Proof of Theorem \ref{t:irregular}}

Let $H$ be the hypergraph from Lemma \ref{l:random} and let $H^*$ be the dual hypergraph of $H$, so
$irr(H^*) = s(H)$. Let $n$ be the number of vertices of $H^*$ that is the number of edges in $H$, i.e.
$n = |E(H)| = \Theta(N^{(r+1)/2}\sqrt{\log N}) $. 
We also have that $s(H) = c N^r$, so for $ \epsilon < 2/(r+1)$ we have
$$
irr(H^*) \geq c N^r \geq C n^{2r/(r+1)}/\log^r n \geq n^{2 - \epsilon}\;.
$$
\qed

\section{Concluding remarks and open problems}

As mentioned in the introduction, there are connections between the considered problem and Sidon sets.
Recall that a $B_h[1]$-Sidon set is a set $X$ of integers such that for any integer $q$, there is at most one subset $X'$ of $X$, $|X'|=h$,  so that the sum of elements from $X'$ is $q$. The following bounds on the sizes of Sidon sets are known: if $X\subseteq [K]$ and $X$ is a $B_h[1]$-Sidon, then
$|X|\leq (h\cdot h! K)^{1/h}(1+o(1))$,  see for example \cite{P,R}. Let $c(h)$ be a constant depending on $h$ only such that $|X|\leq (c(h)K)^{1/h}(1+o(1))$ for any $B_h[1]$-Sidon set $X$, $X\subseteq [K]$. Consider a hypergraph $H$ that is a union of a complete $h$-uniform hypergraph on $N$ vertices and $\binom{N}{h} - N$  isolated vertices. Then $H$ has the same number $n = N^h/h! (1+o(1))$ of vertices and edges and $s(H) \geq  (1/c(h))^h N^h = \Theta(n)$. So, this only gives a linear lower bound on $s(n)=s(n,n)$, much weaker than Theorem \ref{t:snm}.

Note that a similar problem defined on open neighbourhoods of the vertices of a graph is equivalent to the setting we considered on the complement ${\overline G}$ of the graph $G$.  Indeed, if $f$ is a vertex sum-distinguishing labeling of $G$, then the numbers $\sum_{u\not\in N[v]} f(u)$, $v\in V(G)$  are distinct for any two vertices with distinct open neighbourhoods. We see that $V(G)-N[v] = N_{{\overline G}}(v)$, thus the sums considered correspond to the sums over open neighbourhoods in the complement.

Yet another variant of $s(H)$ is its restriction to {\em injective} labelings. Denoting the corresponding parameters by $s_{inj}(H)$ and $s_{inj}(n,m)$ we see, by definition, that $s_{inj}(H) \ge s(H)$
so $s_{inj}(n,m) \ge s(n,m)$. 
If a hypergraph $H'$ is a union of $H$ and all hypergedges consisting of exactly one vertex of $H$,  then $s_{inj}(H) \le s(H')$, and  $H'$ has at most $|E(H)|+|V(H)|$ edges. Thus, Theorem \ref{t:snm} trivially extends to $s_{inj}(n,m)$.
By defining $s^*_{inj}(n)$ similarly and following the steps of Theorem \ref{t:Delta}, one can also show  that $s^*_{inj}(G) \leq (\Delta +2)n$, for any graph $G$ on $n$ vertices and maximum degree $\Delta$.

In this paper, we addressed hypergraphs on $n$ vertices and $m$ hyperedges, for $n\leq m \leq n^{O(1)}$.
It may be of some interest to determine the behavior of $s(n,m)$ when $m$ is larger than a polynomial function of $n$. As mentioned in the introduction, the closer $m$ gets to $2^n$, the closer $s(n,m)$ gets to be linear in $m$.

Finally, it may be of some interest to improve the upper bound $s^*(G) \leq n(\Delta+1)$ given in Theorem \ref{t:Delta} for regimes of $\Delta$ that are significantly less than quadratic.



\begin{thebibliography}{99}
	
\bibitem{AAB} Ahmad, A. and Al-Mushayt, O. and Ba\v{c}a, M., On edge irregularity strength of graphs.
Appl. Math. Comput. 243 (2014), 607--610.

\bibitem{AS} Alon, N. and Spencer J., {\bf The Probabilistic Method}.
John Wiley \& Sons, 2004.

\bibitem{AHPSVW} Axenovich, M., Harant, J.,  Przyby\l o, J., Sot\'ak, R., Voigt, M.,  Weidelich, J.,
A note on adjacent vertex distinguishing colorings of graphs. 
Discrete Appl. Math. 205 (2016), 1--7.

\bibitem{BBLM} Balister, P., Bollob\'as, B., Lehel, J.,  Morayne, M., Random hypergraph irregularity. SIAM J. Discrete Math. 30 (2016), no. 1, 465--473. 

\bibitem{BB} Bartnicki, T.,  Bosek, B.,  Czerwi\'nski, S.,  Grytczuk, J.,  Matecki, G., \.{Z}elazny, W.,  Additive coloring of planar graphs. Graphs Combin. 30 (2014), no. 5, 1087--1098.

\bibitem{BDG} Bhattacharya, B.,  Das, S.,  Ganguly, S., 
Minimum-weight edge discriminators in hypergraphs. 
Electron. J. Combin. 21 (2014), no. 3, Paper 3.18, 19 pp.

\bibitem{BSz} Blokhuis, A.,  Sz\H{o}nyi, T.,
Irregular weighting of 1-designs.
Discrete Math. 131 (1994), no. 1-3, 339--343.

\bibitem{BP} Bollob\'as, B. Pikhurko, O.,
Integer sets with prescribed pairwise differences being distinct. 
European J. Combin. 26 (2005), no. 5, 607--616.

\bibitem{C} Chartrand, G.,  Jacobson, M.,  Lehel, J., Oellermann, O.,  Ruiz, S.,  Saba, F.,  Irregular networks. 250th Anniversary Conference on Graph Theory,  Congr. Numer. 64 (1988), 197--210.

\bibitem{Ga} Gallian, J.,   A dynamic survey of graph labeling. Electron. J. Combin. 5 (1998), Dynamic Survey 6, 43 pp.

\bibitem{G} Guy, R., Sets of integers whose subsets have distinct sums. Theory and practice of combinatorics, 141--154,
North-Holland Math. Stud., 60, Ann. Discrete Math., 12, North-Holland, Amsterdam, 1982.

\bibitem{GJKLS}   Gy\'arf\'as, A., Jacobson, M., Kinch, L.,  Lehel, J., Schelp, R.,  Irregularity strength of uniform hypergraphs. J. Combin. Math. Combin. Comput. 11 (1992), 161--172.

\bibitem{N} Nierhoff, T., 
A tight bound on the irregularity strength of graphs. 
SIAM J. Discrete Math. 13 (2000), no. 3, 313--323.

\bibitem{P} Plagne, A., 
Recent progress on finite $B_h[g]$  sets.
Proceedings of the Thirty-second Southeastern International Conference on Combinatorics, Graph Theory and Computing (Baton Rouge, LA, 2001).
Congr. Numer. 153 (2001), 49--64.

\bibitem{R} Ruzsa, I., Solving a linear equation in a set of integers I, Acta Arith. LXV.3
(1993), 259--282.

\end{thebibliography}
\end{document}